\documentclass[10pt,draft,twoside]{amsart}
\usepackage{amssymb}
\usepackage{latexsym}
\usepackage{amsfonts}
\usepackage{amsmath}

\newtheorem{theorem}{Theorem}[section]
\newtheorem{proposition}[theorem]{Proposition}

\theoremstyle{definition}

\newcommand{\sym}[1]{{\sf Sym}\,#1}

\renewcommand{\wr}{\,{\sf wr}\,}

\newcommand{\aut}[1]{{\sf Aut}\,{#1}}

\newcommand{\func}[2]{\mbox{\sf Func}({#1},#2)}

\renewcommand{\leq}{\leqslant}

\begin{document}

\title[Embeddings of permutation groups in wreath 
products]{Embedding permutation groups\\ into wreath products in product action}
\author{Cheryl E. Praeger and Csaba Schneider}
\address{[Praeger] Centre for Mathematics of Symmetry and Computation\\
School of Mathematics and Statistics\\
The University of Western Australia\\
35 Stirling Highway, Crawley\\
Western Australia 6009\\ \and
[Schneider] Centro de \'Algebra da Universidade de Lisboa\\
Av. Prof. Gama Pinto, 2, 1649-003\\ Lisboa, Portugal
}

\email{cheryl.praeger@uwa.edu.au,
csaba.schneider@gmail.com\protect{\newline} {\it WWW:}
www.maths.uwa.edu.au/$\sim$praeger, www.sztaki.hu/$\sim$schneider}

\begin{abstract}
The wreath product of two permutation groups $G\leq\sym\Gamma$ and 
$H\leq\sym\Delta$ can be considered as a permutation group acting on the set 
$\Pi$ of functions from $\Delta$ to $\Gamma$. This action, 
usually called the product action, of a wreath product
plays a very important role in the theory of permutation groups, as 
several classes of primitive or quasiprimitive groups can be described 
as subgroups of such wreath products.
In addition, subgroups of wreath products in product action arise 
as automorphism groups of graph products and codes. 
In this paper we consider subgroups $X$ of full wreath products 
$\sym\Gamma\wr\sym\Delta$ 
in product action. Our main result is that,  in 
a suitable conjugate of
$X$, the subgroup of $\sym\Gamma$ induced by a stabilizer 
of a coordinate $\delta\in\Delta$ only depends on the orbit of $\delta$ under 
the induced action of $X$ on $\Delta$. 
Hence, if the action of $X$ on $\Delta$ is transitive, then
$X$ can be embedded into a much smaller wreath product. 
Further, if this $X$-action is intransitive, then $X$ can be embedded into 
a direct product of such wreath products where the factors of the direct
product correspond to the $X$-orbits in $\Delta$. We offer an application
of the main theorems to  error-correcting codes in Hamming graphs.
\end{abstract}

\thanks{This paper is dedicated to the memory of Alf Van der Poorten.\\
{\it Date:} draft typeset \today
The first author is supported by Australian Research Council Federation Fellowship FF0776186. The second author acknowledges the support of the grant 
PTDC/MAT/101993/2008 of the  {\em Funda\c c\~ao para a Ci\^ encia e a Tecnologia} (Portugal) and of the Hungarian Scientific Research Fund (OTKA) grant~72845.}

\maketitle

\section{Introduction}\label{s1}

Subgroups of wreath products in product action arise in a number of different contexts. Their importance for group actions is due to the fact that such subgroups give rise to several of the `O'Nan--Scott types' of finite primitive permutation groups (see~\cite[Chapter~2]{dm} or~\cite[Sections~1.10 and~4.3]{cam}) and finite quasiprimitive groups~\cite{p}. They have received special attention recently in the work of Aschbacher~\cite{Asch1,Asch2} aimed at studying intervals in subgroup lattices~\cite{Asch3} (with Shareshian), and of the authors~\cite{recog,transcs,intranscs,PS} investigating invariant cartesian decompositions (with Baddeley). The product action of the wreath product $W=\sym\Gamma\wr \sym\Delta$ is its natural action on the set  $\Pi=\func\Delta \Gamma$ of functions from $\Delta$ to $\Gamma$, described in Subsection~\ref{sub:setup}. If $\Delta=\{1,\dots,m\}$ then $\Pi$ can be identified with the  set  $\Gamma^m$ of ordered $m$-tuples of elements of $\Gamma$, and in this case subgroups of $W$ arise as automorphism groups of various kinds of graph products, as automorphism groups of codes  of length $m$ over the alphabet $\Gamma$ (regarded as subsets of $\Gamma^m$), and  as automorphism groups of a special class of chamber systems in the sense of Tits. To study subgroups $X$ of $W,$ and the structures on which they act, one considers the subgroup $H$ of $\sym\Delta$ induced by $X$ along with the `components' $X^{\Gamma_\delta}$, which are  permutation groups on $\Gamma$, defined in  Subsection~\ref{sub:components}, for each $\delta\in\Delta$. 

We are interested in $X$ up to permutation isomorphism, and wish to replace $X$ by some conjugate in $W$ which gives a simple form with respect to the product action, both for $X$ and the structures on which it acts. This has been done in detail by Kov\'acs in \cite{Kov} in the case where $X$ is primitive on $\Pi$. Kov\'acs also provides a simple form for subgroups inducing a transitive group $H$ on $\Delta$; indeed \cite[(2.2)]{Kov} is the first assertion of Theorem~\ref{embedth2}(b). One way to handle general  subgroups $X$ is to proceed indirectly by appeal to the Embedding Theorem for subgroups of $W$ using a different action, namely its imprimitive action on $\Gamma\times\Delta$ (see for example~\cite[Theorem 8.5]{BMMN}). However this indirect method does not allow us to keep track of important properties of the underlying product structure. For example, if $X$ is an automorphism group of a code $C\subset\Gamma^m$ then we may wish to maintain the property that $C$ contains a specified codeword, say $(\gamma,\dots,\gamma)$ for a fixed $\gamma\in\Gamma$, as well as to obtain a simple form for the group $X$. Thus a direct approach is highly desirable, and the results of this paper provide such an approach. Our main result shows how to choose a form for $X$ so that the $\delta$-component depends only on the $X$-orbit in $\Delta$ containing $\delta$. 

\begin{theorem}\label{embedth2}
Suppose that  $W=\sym\Gamma\wr\sym\Delta$ acts in 
product action on $\func\Delta \Gamma$ with base group $B=\func\Delta {\sym\Gamma}$. 
Let $X\leq W$, 
$\varphi\in\func\Delta\Gamma$ and $\delta_1\in\Delta$. Then the following hold.
\begin{enumerate}
\item[(a)] There is an element $x\in B$ such that the components of $x^{-1}Xx$, as defined in {\rm(\ref{defcomp})}, are constant on each $X$-orbit in $\Delta$. Moreover, if  the $\delta$-component of $X$ is transitive on $\Gamma$ for each $\delta\in\Delta$, then the element $x$ can be chosen to fix $\varphi$.
\item[(b)] If the group $H$ induced by $X$ on $\Delta$ is transitive, and if $G$ is  the $\delta_1$-component of $X$, then the element $x$ may be chosen in $\func \Delta {\sym\Gamma}$ such that $X^x\leq G\wr H$, (and also such that $\varphi x=\varphi$ if $G$ is transitive on $\Gamma$). 
\end{enumerate}
\end{theorem}
  
Note that, in part (b), $G\wr H$ denotes a particular subgroup of $W$ (defined in Subsection~\ref{sub:setup}) and not just an isomorphism class of groups.  
If the subgroup $X$ is transitive on $\Pi$ then all of its components are transitive (Theorem~\ref{transgamma}), so the additional condition on the element $x$ in Theorem~\ref{embedth2} to fix a given point is possible.

\begin{theorem}\label{transgamma}
Let $W=\sym\Gamma\wr\sym\Delta$ act in product action on $\func\Delta \Gamma$ with base group $B=$ $\func\Delta {\sym\Gamma}$, where $\Delta, \Gamma$ are finite sets.
If $X$ is a transitive subgroup of $W$, then each component of $X$ is transitive on $\Gamma$. Moreover, if $X$ acts transitively on $\Delta$ then  each component of the intersection $X\cap B$ is transitive on $\Gamma$.
\end{theorem}

In many instances the group $X$ will be far from transitive on $\Pi$, but may still satisfy some transitivity conditions. We give a simple application of Theorem~\ref{embedth2} in the context of codes. It is most conveniently stated using coordinate notation. So $\Delta=\{1,\dots,m\}$ and the code $C$ is a subset of $\Gamma^m$ with automorphism group being the setwise stabiliser $X$ of $C$ in $W=\sym \Gamma\wr\sym \Delta$. The image of $C$ under some element of $W$ is a code {\em equivalent} to $C$. Equivalence preserves most important properties, such as the {\em minimum distance} $d$ of $C$, which is the minimum number of entries in which distinct elements of $C$ differ. For codes in several interesting families, such as completely transitive codes~\cite{GP} and neighbour transitive codes~\cite{Gi} with minimum distance at least $3$, $X$ is transitive on $\Delta$ and all components of $X$ are 2-transitive on $\Gamma$. For such codes with minimum distance $d$, we show that, for our two `favourite' elements $\gamma, \nu$ of $\Gamma$, there is a code equivalent to $C$ containing both the  $m$-tuple $(\gamma^m)$, and the $m$-tuple $(\nu^d,\gamma^{m-d})$ with the first $d$ entries $\nu$ and the remaining entries $\gamma$, while maintaining a simple form for $X$.

\begin{theorem}\label{cor1}
Let $\Delta=\{1,\dots,m\}$, and let $\gamma,\nu$ be distinct elements of 
$\Gamma$. Suppose that $C\subset \Gamma^m$ has minimum distance $d$, cardinality $|C|>1$, and automorphism group $X\leq W= \sym \Gamma\wr\sym \Delta$ such that $X$ induces a transitive group $H$ on $\Delta$ and some component $G$ of $X$ is $2$-transitive on $\Gamma$. Then there exists $x\in W$ such that the equivalent code $C^x$ has automorphism group $X^x\leq G\wr K$ (with $K$ conjugate to $H$ in $\sym\Delta$) and $C^x$ contains the $m$-tuples $(\gamma^m)$ and $(\nu^d,\gamma^{m-d})$.    
\end{theorem}

\subsection{Wreath products and the product action}\label{sub:setup}

For our proofs, it is most convenient to use `function notation' for defining the wreath product and its product action. 

Let $\Gamma,\ \Delta$ be sets and let $G,\ H$ be subgroups of $\sym \Gamma$, $\sym \Delta$ respectively.  
Set $B=\func\Delta G$, the set of functions from $\Delta$ to $G$. Then $B$ 
is a group with respect to pointwise
multiplication of its elements:  the product of the functions 
$f$ and $g$ is the function $fg$ that maps $\delta\mapsto (\delta f)(\delta g)$. Moreover $B$ is isomorphic to the direct product of $|\Delta|$ copies of $G$ (or the cartesian product if $\Delta$ is infinite): for $\delta\in\Delta$, set 
\begin{equation}\label{defgdelta}
G_\delta=\{f\in\func\Delta G\ |\  \delta'f=1 \mbox{ for all }\delta'\in\Delta\setminus\{\delta\}\}
\end{equation} 
and define the map $\sigma_\delta:\func \Delta G \rightarrow G_\delta$ by
\begin{equation}\label{defsigdel}
 \sigma_\delta:f\mapsto f_\delta\ \mbox{where}\ \delta'f_\delta = \left\{\begin{array}{cc}
\delta f&\mbox{if $\delta'=\delta$}\\
1  &\mbox{if $\delta'\ne\delta$.}\\
\end{array}\right.
\end{equation}
Then $G_\delta$ is a subgroup isomorphic to $G$, $B$ is the direct product of the subgroups $G_\delta$ (the cartesian product if $\Delta$ is infinite), and $\sigma_\delta$ is the natural projection map $G\rightarrow G_\delta$.

We define a homomorphism $\tau$ from $H$ to $\aut B$: 
for $f\in B$ and $h\in H$ let 
$f (h\tau)$ be the function that maps 
$
\delta\mapsto \delta h^{-1}f.
$
Now the {\em wreath product} $G\wr H$ is defined as the semidirect product
$B\rtimes H$ with respect to the homomorphism $\tau$.
The normal subgroup $B$ is called the {\em base group} of the wreath
product, and $H$ is the {\em top group}. A useful and easy computation shows that 
\begin{equation}\label{conjeq}
(\delta h^{-1})f=\delta f^{h}\quad\mbox{for all}
\quad h\in H,\ f\in\func \Delta G,\ \delta\in\Delta.
\end{equation}

The {\em product action} of $G\wr H$ on $\Pi=\func\Delta\Gamma$ is defined as follows.
Let $f\in\func\Delta G$, $h\in H$ and set $g=fh$. For $\varphi\in\Pi$ 
 we define $\varphi g$ as the function that maps 
$\delta\in\Delta$ to 
\begin{equation}\label{prodactdef}
\delta(\varphi g)=(\delta h^{-1}\varphi) (\delta h^{-1} f).
\end{equation}
Note that $\delta h^{-1}\varphi\in\Gamma$, and $\delta h^{-1}f\in\sym\Gamma$.
Thus $(\delta h^{-1}\varphi) (\delta h^{-1} f)\in\Gamma$, and so
$\varphi g\in\func\Delta\Gamma=\Pi$, as required. It is straightforward to verify that this action  of $G\wr H$ on $\Pi$ is well-defined and faithful (see also~\cite[Section~2.7]{dm}).

Let $\gamma$ be a fixed element of $\Gamma$ and let 
$\varphi$ be the element of $\func\Delta\Gamma$ that maps $\delta\mapsto\gamma$
for all $\delta\in\Delta$. Let us compute the stabiliser $(G\wr H)_\varphi$ 
of $\varphi$ in
$G\wr H$. The subgroup $H\leq (G\wr H)_\varphi$ since, if $h\in H$, then 
$
\delta(\varphi h)=(\delta h^{-1}\varphi)=\gamma.
$
Therefore $(G\wr H)_\varphi=B_\varphi H$. 
Suppose that $f\in B$. Then the image of $\delta\in\Delta$ under $\varphi f$
is
$$
(\delta\varphi)(\delta f)=\gamma(\delta f).
$$
Hence $f\in B_\varphi$ if and only if $\delta f\in G_\gamma$ for all 
$\delta\in \Delta$. Thus 
\begin{equation}\label{wreathstab}
(G\wr H)_\varphi=\{fh\ |\ \delta f\in G_\gamma,\ \mbox{for all}\ \delta\in\Delta,\ h\in H\}=
\func\Delta{G_\gamma} H.
\end{equation}

In order to facilitate our discussion of subgroups of wreath products we 
invoke the language of cartesian decompositions which was 
introduced in~\cite{recog}
and was subsequently used in~\cite{transcs,intranscs,PS} to describe
innately transitive subgroups of wreath products in product action.
Consider the set $\Pi=\func\Delta \Gamma$, and define, for each 
$\delta\in\Delta$, 
a partition $\Gamma_\delta$ of $\Pi$ as follows. Set 
\begin{equation}\label{gammadelta}
\Gamma_\delta=\{\gamma_\delta\ |\ 
\gamma\in\Gamma\},\quad \mbox{where}\quad \gamma_\delta:=\{\varphi\in\Pi\ 
|\ \delta\varphi=\gamma\}.
\end{equation} 
It is routine to check that $\Gamma_\delta$ is indeed a partition 
of $\Pi$. 
Our notation reflects two important facts.  Firstly, the map 
$\delta\mapsto \Gamma_\delta$ is a bijection
between $\Delta$ and $\{\Gamma_\delta\ |\ \delta\in\Delta\}$. Secondly,
for a fixed $\delta\in\Delta$,
the map $\gamma\mapsto\gamma_\delta$ is a bijection between $\Gamma$ 
and $\Gamma_\delta$. For $\gamma\in\Gamma$ and $\delta\in\Delta$, the element 
$\gamma_\delta\in\Gamma_\delta$ can be considered as ``the copy'' of $\gamma$ in
$\Gamma_\delta$, and is usually called the 
\emph{$\gamma$-part of $\Gamma_\delta$}.

The cartesian product 
$\prod_{\delta\in\Delta}\Gamma_\delta$ can be bijectively identified with 
the original 
set $\Pi$. 
Namely, choosing $\gamma_\delta\in\Gamma_\delta$, one for each 
$\delta\in\Delta$, the intersection $\bigcap_{\delta\in\Delta}\gamma_\delta$ 
consists of a single point of $\Pi$, and this gives rise to a 
bijection from the cartesian product $\prod_{\delta\in\Delta}\Gamma_\delta$ 
to $\Pi$.
Therefore, 
in the terminology of~\cite{recog}, the set $\{\Gamma_\delta\,|\,\delta\in\Delta\}$ is called
a {\em cartesian decomposition} of $\Pi$. In fact, this set of partitions is 
viewed as the
natural cartesian decomposition of $\Pi$. 
As $\sym \Gamma\wr\sym\Delta$ is a permutation group acting on 
$\Pi$, the action of $\sym \Gamma\wr\sym\Delta$ 
can be extended to subsets of $\Pi$, 
subsets of subsets, etc. Hence one can consider the action of $\sym \Gamma\wr\sym\Delta$
on the set of partitions of $\Pi$. It is easy to see 
that $\{\Gamma_\delta\ |\ \delta\in\Delta\}$ is invariant under this action, and 
we will see that the $(\sym \Gamma\wr\sym\Delta)$-action on this set is permutationally isomorphic to the induced action of $\sym \Gamma\wr\sym\Delta$
on $\Delta$ (defined in Subsection~\ref{sub:components}) under the bijection $\delta\mapsto\Gamma_\delta$.
The natural product action of $\sym \Gamma\wr\sym\Delta$ 
on $\prod_{\delta\in\Delta}\Gamma_\delta$ is permutationally isomorphic to 
its action on $\Pi$, and indeed the stabiliser in $\sym \Pi$ of this 
cartesian decomposition is the wreath product $\sym \Gamma\wr\sym\Delta$.  
See \cite{recog} for a more detailed discussion.

In the case where $\Delta=\{1,\dots,m\}$, it is worth expressing the product action of the wreath product in coordinate notation. View $\func \Delta G$ and $\Pi=\func\Delta\Gamma$ as $G^m$ and $\Gamma^m$,
respectively. Then, for $(\gamma_1,\ldots,\gamma_m)\in\Gamma^m$ and
$(g_1,\ldots,g_m)h\in G\wr H$, we have that
$$
(\gamma_1,\ldots,\gamma_m)((g_1,\ldots,g_m)h)=(\gamma_{1h^{-1}}g_{1h^{-1}},
\ldots,\gamma_{mh^{-1}}g_{mh^{-1}}).
$$

\subsection{Subgroups of wreath products and their components}\label{sub:components}

Let $X\leq\sym\Gamma\wr\sym\Delta$. We define, for $\delta\in\Delta$, the {\em $\delta$-component} $X^{\Gamma_\delta}$ of $X$ as a subgroup of $\sym\Gamma$ as follows. Recall that each element of $X$ is of the form $fh$, where $f\in\func\Delta\Gamma$ and $h\in\sym\Delta$. Recall also the definition of $\Gamma_\delta$ in (\ref{gammadelta}). 
Now $X$ permutes the partitions $\Gamma_\delta$ and we denote the stabiliser
$\{x\in X \,|\, \Gamma_\delta x=\Gamma_\delta\}$ in $X$ of $\Gamma_\delta$ by $X_{\Gamma_\delta}$. Then   $X_{\Gamma_\delta}=\{fh\in X\,|\,\delta h =\delta\}$, and the $\delta$-component $X^{\Gamma_\delta}$ of $X$ is the image of $X_{\Gamma_\delta}$ in $\sym\Gamma$ under the map $fh\mapsto \delta f$, namely
\begin{equation}\label{defcomp}
 X^{\Gamma_\delta} :=\{\delta f\,|\, \mbox{there exists}\ fh\in X_{\Gamma_\delta}\ \mbox{for some $h$}\}. 
\end{equation}
The bijection $\gamma_\delta\mapsto\gamma$ is equivariant with respect to the actions of $X_{\Gamma_\delta}$ on $\Gamma_\delta$ and $X^{\Gamma_\delta}$ on $\Gamma$. Later (when we define the induced action of $X$ on $\Delta$) we will see that $\Gamma_\delta\, x=\Gamma_{\delta x}$, for $x\in X$.  

In order to prove Proposition~\ref{transgamma}, we need more information about  subgroups of $W=\sym\Gamma\wr\sym\Delta$ which do not act transitively on $\Delta$. 
It turns out that such subgroups $X$ may be viewed as subgroups of a direct product in product action: for sets $\Omega_1$ and $\Omega_2$, and permutation groups $G\leq\sym\Omega_1$ and $H\leq\sym\Omega_2$, the {\em product action} of the direct product $G\times H$ is the natural action of $G\times H$ on $\Omega_1\times\Omega_2$ given by $(g,h):(\omega_1,\omega_2)\mapsto (\omega_1g,\omega_2h)$ for $(\omega_1,\omega_2)\in\Omega_1\times\Omega_2$ and $(g,h)\in G\times H$. We construct a {\em permutational embedding} $(\vartheta,\chi)$ of $X$ acting on $\Pi=\func\Delta \Gamma$ into $\sym\Omega_1\times\sym\Omega_2$ acting on  $\Omega_1\times\Omega_2$, by which we mean a bijection $\chi:\Pi\rightarrow \Omega_1\times\Omega_2$ and a monomorphism $\vartheta:X\rightarrow 
\sym\Omega_1\times\sym\Omega_2$ such that, for all $\varphi\in\Pi$ and all $x\in X$, $(\varphi\,x)\chi = (\varphi \chi){x\vartheta}$. 

For a proper non-empty subset $\Delta'$ of $\Delta$, and an element $\varphi\in\func\Delta\Gamma$, define $\varphi|_{\Delta'}\in\func{\Delta'}\Gamma$ as the restriction of $\varphi$ to $\Delta'$, so 
$\delta\,\varphi|_{\Delta'}=\delta\varphi$ for all $\delta\in\Delta'$. 
For $X\leq\sym \Gamma\wr\sym\Delta$, define the induced action of $X$ on $\Delta$ by $fh: \delta\mapsto \delta h$; equivalently this is the action $x:\delta\mapsto\delta\,x$ defined by $\Gamma_\delta\,x=\Gamma_{\delta x}$.

\begin{proposition}\label{intranscd-dirpdt}
Let $W=\sym\Gamma\wr\sym\Delta$, in product action on $\Pi=\func\Delta \Gamma$, and suppose that $X\leq W$, such that $X$ leaves invariant a proper non-empty subset $\Delta_0$ of $\Delta$ in the induced $X$-action on $\Delta$. Let $\Delta_1=\Delta\setminus\Delta_0$, and set $\Omega_0=\func{\Delta_0} \Gamma$ and $\Omega_1=\func {\Delta_1}\Gamma$. Then the following hold.
\begin{enumerate}
\item[(a)] The map $\vartheta:\Pi\rightarrow \Omega_0\times\Omega_1$ defined by $\varphi\vartheta = (\varphi|_{\Delta_0}, \varphi|_{\Delta_1})$, for $\varphi\in\Pi$,  is a bijection.
\item[(b)] The map $\chi:X\rightarrow \sym\Omega_0\times \sym\Omega_1$ defined by $x\chi=(x_0,x_1)$, where  $\varphi|_{\Delta_i}\,x_i=(\varphi x)|_{\Delta_i}$ for $\varphi|_{\Delta_i}\in\Omega_i$,   is a monomorphism. 
\item[(c)] For $i=0,1$, if $\sigma_i: \sym\Omega_0\times \sym\Omega_1\rightarrow \sym\Omega_i$ is the projection map $(x_0,x_1)\sigma_i=x_i$, then $X\chi\sigma_i$ is contained in $W_i:=\sym\Gamma\wr\sym\Delta_i$, and for each $\delta\in\Delta_i$, the $\delta$-components of $X$ and $X\chi\sigma_i$ are the same subgroup of $\sym\Gamma$.
\item[(d)]  $(\vartheta,\chi)$ is  a permutational embedding  of $X$ on $\Pi$ into the group $\sym\Omega_0\times \sym\Omega_1$ in its product action on $\Omega_0\times\Omega_1$, and $X\chi\leq W_0\times W_1$.
\end{enumerate}
\end{proposition}

\begin{proof}
(a) This follows from the definition of the maps $\varphi|_{\Delta_i}$ as restrictions of $\varphi$. 

(b) Let $x\in X$ and $x\chi=(x_0,x_1)$. Note that $\varphi\,x\in\Pi$, for $\varphi\in\Pi$, and hence $(\varphi\,x)|_{\Delta_i}\in\Omega_i$. It is straightforward to check that $\varphi|_{\Delta_i}\mapsto (\varphi\,x)|_{\Delta_i}$ is a bijection $\Omega_i\rightarrow\Omega_i$. Thus $x_i\in\sym\Omega_i$, for each $i$, and the map $\chi$ is well defined. Let also $y\in X$ and $y\chi=(y_0,y_1)$.
It follows immediately from the definition of the $x_i$ and $y_i$ that $(xy)_i=x_iy_i$ for each $i$, and hence that $x\chi y\chi=(xy)\chi$. Thus $\chi$ is a homomorphism. If $x\in\ker\chi$ then, for each $\varphi\in\Pi$ and each $i$, $\varphi|_{\Delta_i}= \varphi|_{\Delta_i}\,x_i=(\varphi x)|_{\Delta_i}$. Thus $\varphi=\varphi x$. Since this holds for all $\varphi\in\Pi$, $x=1$.

(c)  As in (\ref{gammadelta}), for each $\delta\in\Delta_i$ we define a partition
$\Gamma^i_\delta$  of $\Omega_i$ as follows. For $\gamma\in\Gamma$, we define $\gamma^i_\delta=\{\psi\in\Omega_i\,|\,\delta\,\psi=\gamma\}$ and $\Gamma^{i}_\delta=\{\gamma^i_\delta\, |\, 
\gamma\in\Gamma\}$. Since $(\varphi\,x)|_{\Delta_i}=\varphi_{\Delta_i}x_i$ we have $\gamma^i_\delta\,x_i=\gamma_{\delta\,x}^i$ so that $\Gamma^i_\delta\,x_i=\Gamma_{\delta\,x}^i$. Thus $X\chi\sigma_i$ leaves invariant the set of partitions $\{\Gamma^i_\delta\,|\,\delta\in\Delta_i\}$ which forms a cartesian decomposition of $\Omega_i$. Hence $X\chi\sigma_i$ is contained in $W_i$. The stabiliser of $\Gamma^{i}_\delta$ in $X\chi\sigma_i$ is $(X_{\Gamma_\delta})\chi\sigma_i$ and the $\delta$-component of  $X\chi\sigma_i$, defined as in (\ref{defcomp}), is equal to the $\delta$-component $X^{\Gamma_\delta}$ of $X$.

(d) This follows since, for all $\varphi\in\Pi$ and all $x\in X$, we have
\[
\varphi\vartheta\, x\chi =  (\varphi|_{\Delta_0}, \varphi|_{\Delta_1}) x\chi =  ((\varphi x)|_{\Delta_0}, (\varphi x)|_{\Delta_1}) =(\varphi x)\vartheta
\]
and since, by part (c), $X\chi\sigma_i\leq W_i$.
%
\end{proof}

\section{Proof of Theorem~\ref{embedth2}}\label{sec:embedth2}

Suppose that  $W=\sym\Gamma\wr\sym\Delta$ acts in product action on $\Pi=\func\Delta \Gamma$ with base group $B=\func\Delta{\sym\Gamma}$. Let $X\leq W$, 
$\varphi\in\func\Delta\Gamma$ and $\delta_1\in\Delta$. 
Note that $B$ is the kernel of the induced action of $W$ on $\Delta$, so if $x\in B$, then the $X$-orbits in $\Delta$ coincide with the $x^{-1}Xx$-orbits in $\Delta$. For the computations in the proof we often use the properties given in (\ref{conjeq}) and (\ref{prodactdef}), and the equality $\delta (ff')=(\delta f)(\delta f')$,  for $f,f'\in B$, $h\in\sym\Delta$, $\delta\in\Delta$.

Let $\Delta_1,\dots,\Delta_r$ be the $X$-orbits in $\Delta$ under the action induced by $X$ on $\Delta$.
For $1\leq i\leq r$ choose $\delta_i\in\Delta_i$, 
and for each $\delta\in\Delta_i$, choose $t_\delta\in X$ such that $\Gamma_{\delta_i} t_\delta = \Gamma_{\delta}$,
and in particular take $t_{\delta_i}=1$.
Then $t_\delta = f_\delta h_\delta$ with $f_\delta\in B$ and $h_\delta\in\sym\Delta$
such that $\delta_ih_\delta=\delta$.
Also $X_{\Gamma_\delta}=(X_{\Gamma_{\delta_i}})^{t_\delta}$. 

\medskip\noindent\emph{Claim 1:\quad If the $\delta_i$-component is transitive on $\Gamma$, then we may assume in addition that, for each $\delta\in\Delta_i$, $\delta_if_\delta$ fixes the point $\delta\varphi$ of $\Gamma$.}

Since we have $t_{\delta_i}=1$, the element $\delta_if_{\delta_i}$ is the identity of $\sym\Gamma$ and hence fixes $\delta\varphi$.
Let $\delta\in\Delta_i\setminus\{\delta_i\}$ and consider $s_\delta=fh\in X_{\Gamma_{\delta_i}}$ with $f\in B$ and $h\in\sym\Delta$. Then $\delta_ih=\delta_i$, and the element $s_\delta t_\delta$ is equal to $f_\delta'h_\delta'$, with $f_\delta'=ff_\delta^{h^{-1}}$ and $h_\delta'=hh_\delta$, and satisfies  $\Gamma_{\delta_i} s_\delta t_\delta = \Gamma_{\delta}$. Moreover 
$\delta_i f_\delta'=(\delta_if)((\delta_ih)f_\delta)=(\delta_if)(\delta_if_\delta)$, and we note that  $\delta_if\in \sym\Gamma$ lies in the $\delta_i$-component of $X$, see (\ref{defcomp}). If the $\delta_i$-component is transitive on $\Gamma$, then we may choose $s_\delta$ in $ X_{\Gamma_{\delta_i}}$ such that
the element $(\delta_if)((\delta_i)f_\delta)$ fixes $\delta\varphi$. Replacing $t_\delta$ by $s_\delta t_\delta$ gives an element with the required properties.

\medskip\noindent\emph{Claim 2:\quad For $\delta\in\Delta_i$, the $\delta$-component $X^{\Gamma_\delta}$ equals $(X^{\Gamma_{\delta_i}})^{\delta_if_\delta}$.}

Let $\delta_i f\in X^{\Gamma_{\delta_i}}$. By (\ref{defcomp}), there exists $h\in\sym\Delta$ such that
$\delta_i h=\delta_i$ and $fh\in X_{\Gamma_{\delta_i}}$.  Therefore  $X_{\Gamma_\delta}$ contains 
\begin{eqnarray*}
(fh)^{t_\delta}&=&f^{t_\delta} h^{f_\delta h_\delta} = f^{t_\delta}(f_\delta^{-1}f_\delta^{h^{-1}}h)^{h_\delta}
= f^{t_\delta} (f_\delta^{-1})^{h_\delta}f_\delta^{h^{-1}h_\delta}h^{h_\delta}.
 \\
\end{eqnarray*}
This implies that the $\delta$-component $X^{\Gamma_\delta}$ contains 
\begin{eqnarray*}
 \delta (f^{t_\delta}(f_\delta^{-1})^{h_\delta}f_\delta^{h^{-1}h_\delta})&=&
((\delta h_\delta^{-1})f^{f_\delta}) ((\delta h_\delta^{-1})f_\delta^{-1}) ((\delta h_\delta^{-1}h)f_\delta)\\
\end{eqnarray*}
and using the facts that $\delta_i h_\delta = \delta$ and $\delta_i h=\delta_i$, this is equal to
\[
((\delta_i)f^{f_\delta}) ((\delta_i)f_\delta^{-1}) ((\delta_i) f_\delta) = \delta_i(f^{f_\delta}f_\delta^{-1} f_\delta)=\delta_i f^{f_\delta}=(\delta_i f)^{\delta_i f_\delta}.
\]
Thus $X^{\Gamma_\delta}$ contains $(X^{\Gamma_i})^{\delta_if_\delta}$, and a similar argument proves the reverse inclusion. Hence equality holds and the claim is proved.

\medskip\noindent
\emph{Definition of $x$:}\quad Define $x\in B=\func\Delta{\sym\Gamma}$ as the function satisfying, for each $i$ and each $\delta\in\Delta_i$, $\delta x = \delta_i f_\delta^{-1}$. If all components of $X$ are transitive on $\Gamma$ then we assume (as we may by Claim 1) in addition that, for each $i$ and  $\delta\in\Delta_i$, $\delta_if_\delta$ fixes the point $\delta\varphi$, and hence $\delta x=\delta_if_\delta^{-1}=(\delta_i f_\delta)^{-1}$ fixes $\delta\varphi$. Thus in this case $x$ fixes $\varphi$.

\medskip\noindent\emph{Claim 3:\quad The components of $x^{-1}Xx$ are constant on each of the $\Delta_i$.}

Since $x$ acts trivially on $\Delta$, the stabiliser
$(X^x)_{\Gamma_\delta}=(X_{\Gamma_\delta})^x$ for each $\delta\in\Delta$.
Thus $\delta f$ lies in the $\delta$-component $X^{\Gamma_\delta}$ if and only if there exists $h\in\sym \Delta$ such that $fh\in X_{\Gamma_\delta}$ or equivalently, $(fh)^x=f^xx^{-1}x^{h^{-1}}h\in (X^x)_{\Gamma_\delta}$. This implies that the $\delta$-component of $X^x$ contains 
\[
 \delta (f^xx^{-1}x^{h^{-1}}) = \delta(x^{-1}fx^{h^{-1}}) = (\delta x^{-1}) (\delta f)((\delta h)x) = (\delta f)^{\delta x}
\]
 since $\delta h=\delta$. Thus the $\delta$-component of $X^x$ contains $(X^{\Gamma_\delta})^{\delta x}$
and a similar argument proves the reverse inclusion, so equality holds. Now $\delta x=\delta_if_\delta^{-1}=(\delta_i f_\delta)^{-1}$, which by Claim 2 conjugates $X^{\Gamma_\delta}$ to $X^{\Gamma_{\delta_i}}$. Thus
\[
 (X^x)^{\Gamma_\delta} = (X^{\Gamma_\delta})^{\delta x} = (X^{\Gamma_\delta})^{\delta_i f_\delta^{-1}}
=X^{\Gamma_{\delta_i}}
\]
for all $\delta\in\Delta_i$. This completes the proof of Claim 3, and part (a) follows. 

\medskip
To prove part (b) we assume that the group $H$ induced by $X$ on $\Delta$ is transitive, and let $G$ be  the $\delta_1$-component of $X$. From what we have just proved, each component of $X^x$ is equal to $G$. Let $g'$ be an arbitrary element of $X^x$. Then $g'=x^{-1}gx$ for some $g\in X$, and we have $g=fh$ with $f\in B$ and $h\in \sym\Delta$. By the definition of $H$ we have $h\in H$. Also 
\[
 g'=x^{-1}fhx=(x^{-1}fx^{h^{-1}})h = f'h,\ \mbox{say}.
\]
Thus, in order to prove that $g'\in G\wr H$, it is sufficient to prove that, for each $\delta\in\Delta$, 
$\delta f'\in G$.


Let $\delta':=\delta h$. Then $hh_{\delta'}^{-1}h_\delta$ fixes $\delta$, and so
$X_{\Gamma_\delta}$ contains
\[
 gt_{\delta'}^{-1}t_\delta = fhh_{\delta'}^{-1}f_{\delta'}^{-1}f_\delta h_\delta 
=f(f_{\delta'}^{-1}f_\delta)^{h_{\delta'}h^{-1}} hh_{\delta'}^{-1}h_\delta.
\]
Hence $(X^x)_{\Gamma_\delta}=(X_{\Gamma_\delta})^x$ contains
\[
 x^{-1} gt_{\delta'}^{-1}t_\delta x=( x^{-1} f(f_{\delta'}^{-1}f_\delta)^{h_{\delta'}h^{-1}}x^{h_{\delta}^{-1}h_{\delta'}h^{-1}}) hh_{\delta'}^{-1}h_\delta
\]
which equals $f'' hh_{\delta'}^{-1}h_\delta$ say.
This means that the $\delta$-component $G$ of $X^x$ contains
\begin{eqnarray*}
 \delta f''
&=& (\delta x^{-1} f)((\delta hh_{\delta'}^{-1})(f_{\delta'}^{-1}f_\delta))((\delta hh_{\delta'}^{-1}h_\delta)x)\\
&=& (\delta x^{-1}f)((\delta_1(f_{\delta'}^{-1}f_\delta))(\delta x).
\end{eqnarray*}   
By the definition of $x$, $\delta_1(f_{\delta'}^{-1}f_\delta)=(\delta_1 f_{\delta'}^{-1})(\delta_1 f_\delta)
= (\delta' x)(\delta x)^{-1}$. It follows that
\[
 \delta f''=(\delta x^{-1}f)(\delta'x)=(\delta x^{-1}f)(\delta x^{h^{-1}})=\delta(x^{-1}fx^{h^{-1}})=\delta f'.
\]
Therefore $\delta f'\in G$, as required. Thus part (b) is proved, completing the proof of Theorem~\ref{embedth2}.

\section{Proof of Theorem~\ref{transgamma}}\label{sec:prop}

Let $W=\sym\Gamma\wr\sym\Delta$ act in product action on $\Pi=\func\Delta \Gamma$ with base group $B=$ $\func\Delta {\sym\Gamma}$, where $\Delta,\, \Gamma$ are finite sets. Suppose that $X$ is a transitive subgroup of $W,$ and let $K:=X\cap B$.

Let $\delta\in\Delta$, let $\Delta_0$ be the orbit of $X$ in $\Delta$ containing $\delta$, and let $\Delta_1=\Delta\setminus\Delta_0$. By Proposition~\ref{intranscd-dirpdt}~(d), the permutation actions 
of $X$ on $\Pi$ and on $\Omega_0\times\Omega_1$ are equivalent. In particular, as $X$ is transitive on $\Pi$, its projection $X\sigma_0$ is transitive on $\Omega_0$. Further (defining $\Gamma^0_\delta$ as in the proof of Proposition~\ref{intranscd-dirpdt}~(c)), if $X^{\Gamma^0_\delta}$ is transitive then, by Proposition~\ref{intranscd-dirpdt}~(c),  $X^{\Gamma_\delta}$ is transitive. Thus it suffices to prove that all components of $X$ are transitive in the case where $X$ acts transitively on $\Delta$. So assume that $X$ is transitive on $\Delta$. Let $r:=|\Delta|$, and suppose that, for some $\delta\in\Delta$, the $\delta$-component $X^{\Gamma_\delta}$ is intransitive. Now $K$ is a normal subgroup of $X_{\Gamma_\delta}$ and hence, by (\ref{defcomp}),  the $\delta$-component $K^{\Gamma_\delta}$ of $K$ is a normal subgroup of $X^{\Gamma_\delta}$. Hence $K^{\Gamma_\delta}$ has $s$ orbits in its action on $\Gamma$ for some $s>1$. Since $X$ is transitive on $\Delta$ and normalises $K$, it follows that  $K^{\Gamma_\delta}$ has $s$ orbits for each $\delta\in\Delta$. 
Define $L:=\{f\in B\,|\,\delta f\in K^\delta\ \mbox{for each}\ 
\delta\in\Delta\}$.  Then $L\cong\prod_{\delta\in\Delta}K^\delta$, $L$ has $s^r$ orbits in $\Pi$, and $K\leq L\cap X$. Moreover $X$ normalises $L$ and, since $X$ is transitive on $\Pi$, $X$ permutes the $s^r$ orbits of $L$ transitively and $K$ lies in the kernel of this action. Thus $|X/K|$ is divisible by $s^{r}$. However $X/K$ is isomorphic to the transitive group induced by $X$ on $\Delta$ and hence $|X/K|$ divides $r!$. Thus $s^r$ divides $r!$. 
However this is impossible since for any prime $p$ dividing $s$, the order of a Sylow $p$-subgroup of $\sym\Delta$ is at most $p^{r-1}$. Thus $s=1$. This proves both assertions of Theorem~\ref{transgamma}.

\section{Proof of Theorem~\ref{cor1}}\label{sec:cor}

Let $\Delta=\{1,\dots,m\}$, and let $\gamma,\nu$ be distinct elements of $\Gamma$. Suppose that $C\subset \Gamma^m$ has minimum distance $d$, cardinality 
$|C|>1$, and automorphism group $X\leq W= \sym \Gamma\wr\sym \Delta$ such that
 $X$ induces a transitive group $H$ on $\Delta$ and the $1$-component 
$X^{\Gamma_1}$ of $X$ is a $2$-transitive subgroup $G$ of $\sym\Gamma$. In this 
context it is convenient to identify $\Pi=\func \Delta\Gamma$ with $\Gamma^m$, 
and the base group $B$ of $W$ with $(\sym\Gamma)^m$. Under this identification, 
for example, the subgroup $L=\{f\in B\,|\,\delta f\in X^{\Gamma_\delta}\ 
\mbox{for all}\ \delta\in\Delta\}$ of $B$ is identified with the direct product 
$\prod_{\delta\in\Delta} X^{\Gamma_\delta}$ of the components of $X$. Moreover, 
since $X$ acts transitively on $\Delta$, each of the $X^{\Gamma_\delta}$ is 
2-transitive on $\Gamma$.   

Let $a:=(\gamma_1,\ldots,\gamma_m), b:= (\beta_1,\ldots,\beta_m)\in C$ be codewords at distance $d$. Since $G$ is transitive on $\Gamma$, the subgroup $L$ of the base group $B$ is transitive on $\Gamma^m$, so there is an element $x_1\in L$ such that $a\,{x_1}=(\gamma^m)$. Then, since $x_1$ normalises each of the direct factors $X^{\Gamma_\delta}$ of $L$, it follows that $X^{x_1}$ has the same components as $X$. Now we apply Theorem~\ref{embedth2}~(b) and obtain an element $x_2\in B$ such that $X^{x_1x_2}\leq G\wr H$ and  $a\,{x_1x_2}=(\gamma^m)x_2=(\gamma^m)$. Now the image $b\,{x_1x_2}$ differs from $(\gamma^m)$ in exactly $d$ entries. Let $I$ denote this $d$-subset of $\Delta$. Choose $x_3$ in the top group $\sym\Delta$ of $W$ such that $I\,{x_3}=\{1,\dots,d\}$. Then $C\,{x_1x_2x_3}$ contains $a\,{x_1x_2x_3}=(\gamma^m){x_3}=(\gamma^m)$ and $b\,{x_1x_2x_3}$, and the latter $m$-tuple differs from $(\gamma^m)$ precisely in the $d$-subset $I\,{x_3}=\{1,\dots,d\}$. Thus 
entries $d+1,\dots,m$ of  $b\,{x_1x_2x_3}$ are all equal to $\gamma$. The automorphism group $X^{x_1x_2x_3}$ of $C\,{x_1x_2x_3}$ has the same components as $X^{x_1x_2}$ (which are all equal to $G$) and induces the transitive group $K:=H^{x_2}$ on $\Delta$. Thus $X^{x_1x_2x_3}\leq G\wr K$.
Finally, since $G$ is 2-transitive on $\Gamma$, for each $i\leq d$ there is an element $y_i\in G_\gamma$ which maps the $i^{th}$ entry of $b\,{x_1x_2x_3}$ to $\nu$. Let $x_4\in \func\Delta {G_\gamma} \leq B$ be any element such that $ix_4=y_i$ for $i=1,\dots,d$, and set $x=x_1x_2x_3x_4$. Then $X^x\leq G\wr K$ and $C\,x$ contains $(\gamma^m)$ and $(\nu^d,\gamma^{m-d})$.

\section[]{Acknowledgement}
The authors are grateful to an anonymous referee for suggestions leading to clarification and improvement of the exposition.

\end{document}